\documentclass[notitlepage,12pt,oneside]{amsart}

\title[Uniform Asymptotics for Associated Stirling Numbers]{Uniform convergence of an asymptotic approximation to associated Stirling numbers}

\author{E.~Rodney~Canfield}
\email{ercanfie@uga.edu}
\address{University of Georgia, Athens}
\author{J.~William~Helton}
\email{helton@math.ucsd.edu}
\address{University of California, San Diego}
\author{Jared~A.~Hughes}
\email{jahughes@ucsd.edu}
\address{University of California, San Diego}

\date{\today}

\usepackage[top=0.5in,bottom=0.7in]{geometry}

\usepackage{graphicx}
\counterwithin{figure}{section} 
\usepackage{bm}
\usepackage[dvipsnames]{xcolor}
\usepackage{imakeidx}

\usepackage[bookmarks=false]{hyperref}

\usepackage[style=alphabetic]{biblatex}
\usepackage{amsmath, amsthm, amssymb}
\usepackage{mathtools}
\usepackage[justification=raggedright]{caption}
\usepackage{cleveref}

\usepackage{mathrsfs}

\usepackage{float}

\numberwithin{equation}{subsection}
\newtheorem{theorem}{Theorem}[subsection]

\newtheorem{lemma}[theorem]{Lemma}
\newtheorem{corollary}[theorem]{Corollary}

\newtheorem{conjecture}[theorem]{Conjecture}

\theoremstyle{definition}

\theoremstyle{definition}

\theoremstyle{remark}

\DeclareRobustCommand{\curlyStirling}{\genfrac\{\}{0pt}{}}

\newcommand{\abs}[1]{\left\lvert #1\right\rvert}

\def\bs{\bigskip}

\def\Qinv{{\xi_r(p/q)}}

\newcommand\cC{\mathscr{C}}

\begin{document}

\begin{abstract}

Let $S_r(p,q)$ be the $r$-associated Stirling numbers of the second kind, the number of ways to partition a set of size $p$ into $q$ subsets of size at least $r$. For $r=1$, these are the standard Stirling numbers of the second kind, and for $r=2$, these are also known as the Ward Numbers. This paper concerns, asymptotic expansions of these Stirling numbers;
such expansions have been known for many years
(\cite{moserWyman} \cite{Hennecart}).
However, while uniform convergence of these expansions was conjectured in \cite{Hennecart}, it has not been fully proved. A recent paper
\cite{stirlingConnDobro} went a long way, by proving 
uniform convergence on a large set.
In this paper we build on that paper and prove convergence ``everywhere.''

\end{abstract}

\maketitle

\section{Introduction}

We begin by stating preliminaries to define the
Hennecart 
Stirling approximation whose uniform convergence we shall be considering.
Then this section states our main uniform convergence theorem, \Cref{thm:mainUniform}.

The proof
of the main theorem is the subject of \Cref{sec:proofs}.
It depends heavily on the main theorem (Theorem 1.1) of
\cite{stirlingConnDobro}, which proved uniform convergence on a smaller set.

Finally, \Cref{sec:misc} gives 
conjectures as well as some graphs illustrating 
the precision of asymptotic approximations. Also, it provides a new detail,
\Cref{lem:zeroes-of-Br},
which is needed both in our proof and the proof of Theorem 1.1 \cite{stirlingConnDobro}.

\subsection{Definitions}

This presentation starts by introducing some of the building blocks of the formulas central to this paper.

Let $S_r(p,q)$ be the $r$-associated Stirling numbers of the second kind, the number of ways to partition a set of size $p$ into $q$ subsets of size at least $r$. In this paper, we only use $p \geq 1$ and $1 \leq q \leq p/r$. It is standard to define $S_r(p,q)=0$ at all other pairs $(p,q)$, except $S_r(0,0)=1$. 
There are several explicit formulas for $S_r$
which we introduce where they are needed.
A contour integral formula useful for asymptotic theory is
\Cref{eq:Sr-exact-contour-integral}, and a formula suited to
exact integer calculations for $r=2$ is given in \Cref{eq:alekseyev-stirling}.

Now we turn to the main ingredients of 
Hennecart's asymptotic formula for $S_r$.
Define 
\begin{align*}
    B_r(z) \coloneqq e^z - \sum_{k=0}^{r-1} z^k/k!, \qquad Q_r(z) \coloneqq \frac{z B'_r(z)}{B_r(z)}.
\end{align*}
Note as $z\to 0$,
\begin{align*}
    B_r(z) = \frac{z^k}{k!} + O(z^{k+1})
\end{align*}
It is well-known that $Q_r$ is invertible for $z>0$ (see \cite{stirlingConnDobro} Lemma 3.2), and in \Cref{lem:B-ratio-increasing}, we prove something a more general (but with less quantitative power than Lemma 3.2 of \cite{stirlingConnDobro}); hence we can define an inverse function $\xi_r\colon (0,r) \to (0,\infty)$ as
       $$\xi_r = Q_r^{-1}.$$
\cite{Hennecart} and \cite{stirlingConnDobro} 
rely on the 
function
\begin{align*}
    z_0 \coloneqq Q_r^{-1}\left(\frac p q\right) =  \xi_r\left(\frac p q \right).
\end{align*}

\subsection{Our main theorem; convergence of the 
Hennecart approximation formula}

The approximation to $S_r$ provided in 
\cite{Hennecart} is with
\begin{align}
    F_r(p,q) = \frac{p!}{q!(p-rq)!} \left(\frac{p-rq}{e}\right)^{p-rq} \frac{(B_r(z_0))^q}{z_0^{p+1}} \sqrt{\frac{q t_0}{\Phi''(z_0)}},
    \label{eq:def-Hennecart-Fr}
\end{align}
where $qt_0 = p-rq$ and $\Phi(z) = -p \ln(z) + q \ln B_r(z)$, recalling $z_0=\Qinv$. Define
\begin{align}
    H_r (z) &\coloneqq \frac{B_{r-1}(z) B_r (z) + z B_{r-2}(z) B_r(z) - z B_{r-1}^2 (z)}{2 B_r^2 (z)}
\end{align}
Then $\Phi''(z_0) = 2q H_r(z_0)/z_0$.

The main result of this paper asserts uniform convergence of this approximation, as conjectured in \cite{Hennecart}.

\begin{theorem}
    \label{thm:mainUniform}
    For all $r\geq 1$, the asymptotic approximation
    \begin{align}
        S_r(p,q) \sim F_r(p,q) \label{eq:hennecart-approx-original}
    \end{align}
    converges as $p\to\infty$, uniformly for $1\leq q < p/r$. Precisely, 
    \begin{align*}
        \abs{\frac{F_r(p,q)}{S_r(p,q)} - 1} \leq E_r(p)
    \end{align*}
    for some function $E_r(p)$ that depends on $r$ and $p$ only, and for each $r$ we have $E_r(p) \to 0$ as $p \to \infty$.
\end{theorem}

\begin{proof} 
The proof takes the remainder of this paper and culminates  in
\Cref{sec:proof-mainUniform}.
\end{proof}

\section{Proofs}
\label{sec:proofs}

Our proof depends heavily on the proof of the main theorem 
of \cite{stirlingConnDobro}, so we begin by stating 
it in our notation.
Then in \Cref{lem:CD-unif-convergent}, we state (and confirm) a variant of it which has the weaker hypothesis 
$$qz_0= q \Qinv \to \infty \qquad \text{as} \qquad p \to \infty. $$
In \Cref{lem:qz0-to-infty-except-large-q-precise},
we show that $q \Qinv \to \infty$ indeed holds under the assumption $p-rq = \Omega(p^{1/5})$.

Next, in \Cref{lem:full-convergence-large-q}, we prove an approximation that holds when $p-rq=o(p^{2/5})$ by an entirely different approach than \cite{stirlingConnDobro}. \Cref{lem:Rod-approx-eq-Henn} shows this new approximation is indeed uniformly equivalent to Hennecart's formula.

These two regions ($p-rq = \Omega(p^{1/5})$ and $p-rq=o(p^{2/5})$) overlap, hence we get uniform convergence of the Hennecart approximation as claimed by \Cref{thm:mainUniform}. The details of this are handled in \Cref{sec:proof-mainUniform}.

\subsection{The CD approximation formula and CD's theorem on its convergence}
\label{sec:CD-approx-formula}

Closely related to the Hennecart approximation
is a formula introduced in \cite{stirlingConnDobro}. They prove its convergence in a broad range, and our proofs build on it. 

The CD approximation (same form as from \cite{stirlingConnDobro}) is based on the function
$C_r$ defined by
\begin{align}
    C_r(p,q) &\coloneqq F_r(p,q) (p-rq)! \left(\sqrt{2\pi(p-rq)}\left(\frac{p-rq}{e}\right)^{p-rq}\right)^{-1} \label{eq:def-CD-formula}
    \\ &= \frac{p!(B_r(z_0))^q}{2q!z_0^p \sqrt{q\pi z_0 H(z_0)}}
\end{align}

We shall use heavily the
main theorem of \cite{stirlingConnDobro}
and so we now state it.

\begin{theorem}(Theorem 1.1 from \cite{stirlingConnDobro})
Let $r$ be a fixed positive integer. The asymptotic approximation 
\begin{align}
    \label{eq:CD-approx}
    S_r(p,q)\sim F_r(p,q)
\end{align}
converges as $p\to \infty$ uniformly for all $\delta_1 p < q < (1 - \delta_2)p/r$, where $p$ and $q$ are integers, and $\delta_1$ and $\delta_2$ are any positive constants.
\end{theorem}

The goal of this paper is to
remove the need for $\delta_1$ and $\delta_2$. We begin by showing CD's result holds as long as $qz_0\to\infty$, even without the inequality $\delta_1 p < q < (1 - \delta_2)p/r$.

Now we observe that the CD approximation and Hennecart's formula converge to each other, with the only difference between them being the ratio between $(p-rq)!$ and its Stirling approximation:
\begin{lemma}
    \label{lem:CD-matches-Hennecart}
    
    For all $r\geq 1$ and sequence $h(p)$ with $h(p) \to \infty$ as $p\to \infty$, we have
    \begin{align}
        F_r(p,q) \sim C_r(p,q)
    \end{align}
    is uniformly convergent for $1 \leq q \leq (p-h(p))/r$ as $p\to\infty$.
\end{lemma}
\begin{proof}
    The conversion needed here between $C_r$ (CD's formula) and $F_r$ (Hennecart's formula) is due to \cite{stirlingConnDobro}
    (after their Equation 4.1); they perform algebraic manipulation discussed in \Cref{sec:CD-approx-formula}. By definition of $C_r$ in terms of $F_r$ in \Cref{eq:def-CD-formula},
    \begin{align}
        \frac{C_r(p,q)}{F_r(p,q)} &= (p-rq)! \left(\sqrt{2\pi(p-rq)}\left(\frac{p-rq}{e}\right)^{p-rq}\right)^{-1}.  \label{eq:CD-FR-ratio}
    \end{align}
    Note \Cref{eq:CD-FR-ratio} is equal to the ratio of $(p-rq)!$ and the Stirling approximation to $(p-rq)!$, so asymptotically as $p-rq\to\infty$,
    \begin{align}
        \frac{C_r(p,q)}{F_r(p,q)} &= 1 + O\left(\frac{1}{p-rq}\right),
    \end{align}
    which is uniformly convergent to $1$ as $p-rq \to \infty$. Note $p-rq \geq h(p)$ and $h(p)\to\infty$ by assumption, so $p-rq\to\infty$.
\end{proof}

\subsection{Preparatory lemmas on the behavior of $Q_r$}

This subsection provides results on the behavior of $Q_r = x B_r'(x)/B_r(x)$. 

CD Lemma 3.2 shows that $1/(r+1) \leq Q_r'(z) \leq 1$ for all $z > 0$, hinging on the behavior of $\frac{z^{r-1}}{(r-1)! B_r(z)}$. 
While this suffices for our needs, we include 
a result that is less quantitative (it merely shows $Q_r'(z) > 0$) but works for a general class of functions.
Its proof also is based on a new idea.

\Cref{lem:B-ratio-increasing} shows $Q_r(x)$ is strictly increasing at each $x \in [0,\infty)$. Then \Cref{lem:QTo0} and \Cref{lem:QInvLimiting} state asymptotic behavior of $Q_r$ and $Q_r^{-1}$.

\begin{lemma}
    \label{lem:B-ratio-increasing}

    Suppose that
    \begin{align}
        B(x) = \sum_{j\geq 0} a_j x^j
    \end{align}
    is an entire function with $a_j \geq 0$, and at least two of the coefficients are strictly positive. Then
    \begin{align}
        \frac{x B'(x)}{B(x)}
    \end{align}
    is entire and is a strictly increasing function over $x \in [0,\infty)$.
        
    In particular, for each positive integer $r$, we have 
    $Q_r(x) = x B_r'(x) / B_r(x)$ is strictly increasing at each $x\in [0,\infty)$.
\end{lemma}
\begin{proof}
Let $\Theta$ be the operator $(x\,d/dx)$.  Then
$$
\Theta B(x) = \sum a_j j x^j
$$
and
$$
(\Theta)^2 B(x) = \sum a_j j^2 x^j
$$
The product $\Theta B(x) \times \Theta B(x)$ is the sum over certain pairs $(i,j)$
of
$$
a_i a_j ij x^{i+j}.
$$
The product $B(x) \times (\Theta)^2 B(x)$ is the sum over the same pairs $(i,j)$
of
$$
a_i a_j i^2 x^{i+j}.
$$
The product $B(x) \times (\Theta)^2 B(x)$ is the sum over the same pairs $(i,j)$
of
$$
a_i a_j j^2 x^{i+j}.
$$

\bs

\noindent It follows that the difference
\begin{align}
[ B(x) \times (\Theta)^2 B(x) ] ~-~ [\Theta B(x) ]^2
\label{eq:B-difference}
\end{align}
is the sum over the same set of pairs $(i,j)$ of
$$
(1/2) \, a_i a_j [ i^2 + j^2 - 2ij ] x^{i+j}
$$
By hypothesis we have $a_i a_j \neq 0$ for some $i\neq j$, so
the difference (\Cref{eq:B-difference}) is strictly positive for $x\in(0,\infty)$.

\bs

\noindent It remains only to note
\begin{align*}
x \, \frac{d}{dx} \, \frac{x B^{\prime}(x)}{B(x)}
~~ &= ~~
\frac{xB^{\prime}}{B} ~+~ \frac{x^2 B^{\prime\prime}}{B}
~-~ \frac{x^2 (B^{\prime})^2}{B^2} \\
~~ &= ~~
\frac{ [ B(x) \times (\Theta)^2 B(x) ] ~-~ [\Theta B(x) ]^2 } { B^2 }.
\end{align*}
This finishes the proof of the main lemma. To prove the last assertion, note
    \begin{align}
        B_r(z) = \sum_{k=r}^\infty z^k/k!
    \end{align}
    is entire with at least two nonzero coefficients, so the 
    first part of the lemma implies
    \begin{align}
        Q_r(z) = \frac{z B_r'(z)}{B_r(z)}
    \end{align}
    is increasing.
\end{proof}

Now we provide results on the asymptotics of $Q_r$ and $Q_r^{-1}$ as well as the behavior near the point $Q_r(0)=r$.

\begin{lemma}
    \label{lem:QTo0}
    Fix an integer $r\geq 1$. Then
    \begin{align*}
         Q_r(z) &= (r - 1)(1 - 1/z + O(1/z^2)), && \quad (z\to -\infty)
    \\  Q_r(z) &= r + \frac{z}{1 + r} + O(z^2), && \quad (z\to 0)
    \\ Q_r(z)  &= z (1 + O(z^{r-1}/e^z)) = z (1 + O(1/z)), && \quad (z\to+\infty)
    \end{align*}
\end{lemma}

\begin{corollary}
    \label{lem:QInvLimiting}
    \begin{align*}
        Q_r^{-1}(x) &= \frac{-(r-1)}{x-(r-1)} + O(x-(r-1)), && \quad (x \to r-1)
    \\   Q_r^{-1}(x) &= (r+1)(x-r)(1 + O(x-r)), && \quad (x\to r)
    \\ Q_r^{-1}(x) &= x (1 + O(1/x)), && \quad (x\to+\infty)
    \end{align*}
\end{corollary}
\begin{proof}
We just prove the lemma, since the corollary follows directly from it.

    The second equation is obvious from dividing Taylor series expansions
    \begin{align*}
\frac{B_r^{\prime}(z)}{B_r(z)}
 &= 
\frac{(z^{r-1}/(r-1)! + (z^{r}/(r)! + \cdots}
     {(z^{r  }/(r  )! + (z^{r+1}/(r+1)! + \cdots} \\
 &= 
\frac{r}{z} \, \frac{1+z/r+O(z^2)}{1+z/(r+1)+O(z^2)} \\
 &= 
\frac{r}{z} \, \left(1 + z/(r(r+1)) + O(z^2) \right) 
\end{align*}
    
    For the third equation, since $r\geq 1$, we have $B'_r(z) = B_{r-1}(z)$, so
    \begin{align}
        Q_r(z) &= \frac{zB_{r-1}(z)}{B_r(z)}
        \\ &= \frac{z(B_r(z) + \frac{z^{r-1}}{(r-1)!})}{B_r(z)}
        \\ &= z \left (1  + \frac{z^{r-1}}{(r-1)! B_r(z)}\right) \label{eq:third-eq-precursor}
    \end{align}
    Note $B_r(z) = \Theta(e^z)$ as $z\to\infty$. Hence $Q_r(z) = z (1 + O(z^{r-1} / e^z))$ as $z\to\infty$, proving the third equation.

    The first equation follows from \Cref{eq:third-eq-precursor} by noting 
    $$B_r(z) = - z^{r-1}/(r-1)! - z^{r-2}/(r-2)! + O(z^{r-3}) \qquad \text{as} \qquad z\to-\infty.$$
\end{proof}

\subsection{CD's main result only requires $qz_0\to\infty$}

While the following theorem is not directly stated in \cite{stirlingConnDobro}, their proof of CD Theorem 1.1 can be repurposed to prove the following stronger theorem.

\begin{theorem}
    \label{lem:CD-unif-convergent}
    Fix $r\geq 1$. Over $p\geq 1$ and integers $1\leq q< p/r$, if $qz_0\to\infty$ with $z_0 = \Qinv$, then    
    \begin{align}
        S_r(p,q) = C_r(p,q) (1 + O((qz_0)^{-1})).
    \end{align}
    In particular, if $p - rq = \Omega(p^{\delta_7})$ for some $0<\delta_7<1$, then
    \begin{align}
        S_r(p,q) = C_r(p,q) (1 + O(p^{-\delta_7})), \label{eq:Sr-matches-Cr-particular}
    \end{align}
        and
    \begin{align}
        S_r(p,q) = F_r(p,q) (1 + O(p^{-\delta_7})). \label{eq:Sr-matches-Fr-particular}
    \end{align}
\end{theorem}

\begin{proof}
    Their assumption $\delta_1 p \leq q \leq (1/r - \delta_2)p$ can be replaced by $qz_0\to\infty$, with the proofs in \cite{stirlingConnDobro} essentially still working as written. However, we provide some clarification into the details to help a reader who would like to check this claim for themselves. To be consistent with the rest of our paper, we use the variables $p$ and $q$, while the notation in \cite{stirlingConnDobro} lets $n=p$ and $m=q$.

    With only the $qz_0\to\infty$ assumption, certain properties assumed in \cite{stirlingConnDobro} do not necessarily hold. In particular, $z_0$ is not uniformly bounded away from zero since $z_0 \to 0$ as $q/p \to 1/r$ by \Cref{lem:QInvLimiting}. Additionally, $q z_0 = \Theta(q)$ does not necessarily hold. Thus a reader must replace references to $O(p^{-1})$ with $O((q z_0)^{-1})$.

\bs

    We conclude this proof by giving more detail into why $qz_0\to\infty$ is sufficient for CD's proof to work. The workhorse CD Lemma 2.1 relies on CD Lemmas 3.1, 3.2, and 3.3 in its proof.

    Since CD Lemmas 3.1 and 3.2 
    do not involve limiting behavior, they work regardless of assumptions on $p$ and $q$.    CD Lemma 3.3 relies on $qz_0 \to \infty$ to ensure (in their notation) $(h + h^3) \zeta \to 0$, where $h = \Theta( (qz_0)^{1/8})$ and $\zeta = (qz_0)^{-1/2}$. 

    CD start the proof of CD Lemma 2.1 with an exact contour integral form, given by
    \begin{align}
        \label{eq:Sr-exact-contour-integral}
        S_r(p,q) &= \frac{p!}{q!} \frac{1}{2\pi i} \int_C \frac{(B_r(z))^m}{z^{n+1}}dz,
    \end{align}
    where $C$ is a circle about the origin. After a change of variables on page 31, CD split the integral into three parts:
    \begin{align}
        S_r(p,q) = A \int_{-\pi}^\pi \exp(q g(\theta,R)) d\theta = \int_{-\epsilon}^\epsilon + \int_{\epsilon}^\pi + \int_{-\pi}^\epsilon.
    \end{align}

    Expansions used later to handle
    $\int_{-\epsilon}^\epsilon$ are treated first on CD page 30. There, they do not use the $\delta_1,\delta_2$ assumption, so our weaker hypothesis is not challenged.

    One thing to note is that their argument on page 30 makes use of a fact that is not thoroughly proved. We include a full proof of the fact in \Cref{lem:zeroes-of-Br}.
     
    Next, on page 31, CD show the remaining portion of the integral goes to zero by deriving the estimate
    \begin{align*}
        \abs{J} = \abs{\int_{\epsilon}^\pi} = \abs{\int_{-\pi}^\epsilon} \leq e \pi \exp(-(qz_0)^{1/4}) = O((qz_0)^{-1}).
    \end{align*}
    Since $qz_0\to\infty$, we get $|J|\to 0$ as $p\to\infty$. CD also uses $qz_0\to\infty$ to show $\zeta = (qz_0)^{-1/2}$ lies within the domain of convergence of a particular summation $\sum_{k=0}^\infty b_k \zeta^k$ for sufficiently large $p$.
    
    The remainder of the proof of CD Lemma 2.1 turns the formula
    \begin{align}
        S_r(p,q) = \frac{A}{qz_0 H_r(z_0)}\left(\sum_{k=0}^{s-1} \int_{-h}^h (\exp(-\eta^2)b_k d\eta) \zeta^k + O(\zeta^s)\right)
    \end{align}
    into an asymptotic series. 
    They take $s=2$, truncating the sum to only the $k=0$ and $k=1$ terms. For odd $k$, CD notes $b_k$ is a polynomial containing only odd powers of $\eta$, so
    \begin{align}
        \int_{-h}^h \exp(-\eta^2)b_k \zeta^k d\eta = 0
    \end{align}
    In particular, the $k=1$ term of the sum is zero, and $b_0=1$, so
    \begin{align}
        S_r(p,q) &= \frac{A}{qz_0 H(z_0)} \left(\int_{-h}^h \exp(-\eta^2) d\eta + O(\zeta^2)\right)
    \end{align}
    Using $h=\Theta((qz_0)^{1/8})$ and $\zeta=(qz_0)^{-1/2}$:
    \begin{align}
        S_r(p,q) &= \frac{A}{qz_0 H(z_0)} \left(\int_{-\infty}^\infty \exp(-\eta^2) d\eta + O(\zeta^2)\right)
        \\ &= \frac{A \sqrt{\pi}}{qz_0 H(z_0)} (1 + O((qz_0)^{-1}))
    \end{align}
    CD Equation 4.1 is obtained after truncating the summation given in CD Lemma 2.1 to the first term, which we have just verified has relative error given by $O((qz_0)^{-1})$.
    
    This checks that all bounds and approximations involved in the proof of CD's Theorem hold at the level of generality asserted in our theorem (\Cref{lem:CD-unif-convergent}). All other calculations are exact, so this finishes the proof of the main part of our theorem.

    For the later parts of our theorem, suppose $p - rq = \Omega(p^{\delta_7})$ for $0<\delta_7<1$. Then \Cref{lem:qz0-to-infty-except-large-q-precise} implies $qz_0 = \Omega(p^{\delta_7})$. This proves \Cref{eq:Sr-matches-Cr-particular}, and \Cref{lem:CD-matches-Hennecart} finishes the proof of \Cref{eq:Sr-matches-Fr-particular}. 
\end{proof}

\subsection{$qz_0\to\infty$ holds for all 
but very large $q$}

\begin{lemma}
    \label{lem:qz0-to-infty-except-large-q-precise}    
    Fix any $r\geq 1$. For all $\epsilon>0$, as $p\to\infty$ and $1\leq q$, we have
    \begin{align*}
        qz_0 = \Omega(\min\{p^{1-\epsilon}, p-rq\}),
    \end{align*}
    recalling $z_0 = \Qinv$. In particular, if $p-rq = \Omega(p^{\delta_7})$ for $0 < \delta_7 < 1$, then uniformly in $q$, as $p\to\infty$, we have $qz_0 \to \infty$; more precisely, 
    \begin{align*}
        qz_0 = \Omega(p^{\delta_7}).
    \end{align*}
\end{lemma}
\begin{proof}
    \emph{Low $q$}: 
    when $1 \leq q \leq p^{1-\epsilon}$,
    \begin{align}
        Q_r^{-1}(p/q) &= \frac{p}{q} (1 + O(q/p)) \\
        qz_0 = qQ_r^{-1}(q/p) &= p (1 + O(q/p)) = \Theta(p)
    \end{align}
    Since $p \geq p^{1-\epsilon}$, we see $qz_0 = \Omega(p^{1-\epsilon})$. If $\epsilon = 1-\delta_7$, then $qz_0 = \Omega(p^{\delta_7})$.
     
     \bs
    Fix $0<\delta_2<1$.
    
    \emph{Middle $q$}: when $p^{1-\epsilon} \leq q \leq (1-\delta_2)p/r$, let $\delta_6 = Q_r^{-1}(  r/ (1-\delta_2) ) > 0$. Since $p/q \geq r/(1-\delta_2)$ and $Q_r$ is increasing, we have
    $$
    z_0 = Q_r^{-1}(p/q) \geq \delta_6 > 0, \ \ \text{so} \ \ qz_0 \geq \delta_6 q = \Theta(q) \geq\Omega(p^{1-\epsilon}). 
    $$
    If $\epsilon = 1-\delta_7$, then $qz_0 = \Omega(p^{\delta_7})$.

    \bs 
    \emph{High $q$}: when $(1-\delta_2)p/r \leq q < p/r$, then $p/q - r = O(\delta_2)$, so a series expansion from \Cref{lem:QInvLimiting} implies
    \begin{align}
        z_0 &= Q^{-1}(p/q) = (r+1)(p/q - r)(1 + O(p/q - r)). \label{eq:expansion-z0}
        \\ qz_0 &= (r+1)(p - rq)(1 + O(\delta_2)).
    \end{align}
    So $qz_0 = \Theta(p-rq)$ at high $q$. In particular, if $p-rq = \Omega(p^{\delta_7})$, then $qz_0 = \Omega(p^{\delta_7})$.
\end{proof}

\subsection{Very large $q$}
\label{sec:large-q-henn-uniform}

In the previous subsection we obtained our goal 
for $q$ in a certain region which stays well below
$p/r$. In this section we obtain the same goal
for ``very large $q$," meaning $q$ is close to its maximum possible, $p/r$.  
 
 Define $a$
by the equation $p=rq+a$; throughout this sub-section we assume $q\rightarrow\infty$ and $0<a=o(q^{2/5})$.  
We shall prove a new asymptotic formula \Cref{eq:SrBigq} for $S_r(rq+a,q)$, valid in this
range; and then prove that our formula agrees, asymptotically, with that of
Hennecart.

\begin{lemma}
    \label{lem:full-convergence-large-q}

    Fix any integer $r\geq 1$. Then for any positive constant $\delta_8>0$, we have
    \begin{align}
        \label{eq:SrBigq}
        S_r(rq+a,q) ~~ \sim ~~ \frac{(rq+a)!}{a! q! (r!)^q} \, \left(\frac{q}{r+1}\right)^a
    \end{align}
    converges as $q\to\infty$, uniformly for $0\leq a \leq q^{2/5-\delta_8}$.
\end{lemma}

\proof

We start
with two observations.

\bs

\noindent Observation 1: The number of partitions
of a $p$-element set having $\mu_i$  blocks of size $i$ is
$$
\frac{p!}{\prod_i (i!)^{\mu_i}\mu_i!} \, .
$$  
Observation 2: Suppose $\lambda=1^{\mu_1}\,2^{\mu_2}\,\cdots$ is an integer partition of $a$
(notation: $\lambda\vdash a$)
$$
a ~~=~~ \mu_1+2\mu_2+\cdots
$$
having $k$ parts
$$
k ~~=~~ \mu_1+\mu_2+\cdots
$$
and that $k\le q$.  Then
$$
r^{q-k} \, (r+1)^{\mu_1} \, (r+2)^{\mu_2} \, \cdots
$$
is a partition of $rq+a$ into $q$ parts all of which are at least $r$. This correspondence
is reversible.  As a consequence of these two observations, we have the formula
$$
S_r(rq+a,q)
~~ = ~~
\sum_{\substack{ \lambda\vdash a  \\
                 \lambda=1^{\mu_1} 2^{\mu_2} \cdots \\
                 k \le q
               }
     }
\frac{ (rq+a)! } {
                   (r!)^{q-k} (q-k)! \prod_{i\ge 1}((i+r)!)^{\mu_i}\mu_i!
                 }\, , ~~~(k=\mu_1+\mu_2+\cdots).
$$

\bs

\noindent If $q\rightarrow\infty$ and $a=o(q^{1/2})$, then the third condition
$k\le q$ is superfluous, and 
$$
(q-k)! ~~ = ~~   \left(1 ~+~ o(1) \right) \, q^{-k} \, q! 
$$
uniformly over all possible $\lambda$.  This permits:
\begin{align}
S_r(rq+a,q)
~~ &= ~~ \left( 1 ~+~ o(1) \right) ~ \frac{ (rq+a)! }{(r!)^q q!} ~
\sum_{\substack{ \lambda\vdash a  \nonumber\\
                 \lambda=1^{\mu_1} 2^{\mu_2} \cdots
               }
     }
\prod_{i\ge 1} \frac { (r!\,q/(i+r)!)^{\mu_i} } { \mu_i! } \nonumber\\
   &= ~~ \left( 1 ~+~ o(1) \right) ~ \frac{ (rq+a)! }{(r!)^q q!} \,
[x^a] \, \prod_{i\ge 1}\sum_{j=0}^{\infty} 
                                    \frac{ (r!\,qx^i/(i+r)!)^j } { j! } \nonumber\\
   &= ~~ \left( 1 ~+~ o(1) \right) ~ \frac{ (rq+a)! }{(r!)^q q!} \,
[x^a] \, \prod_{i\ge 1} \exp(r!\,qx^i/(i+r)!) \nonumber\\
   &= ~~ \left( 1 ~+~ o(1) \right) ~ \frac{ (rq+a)! }{(r!)^q q!} \,
[x^a] \,  \exp\left(\frac{r!\,q}{x^r} B_{r+1}(x) \right) \label{eq:Sr-as-coeff-exp} 
\end{align}
uniformly  for $a=o(q^{1/2})$.  We are using here the notation
$[x^a]G(x)$ for the coefficient of $x^a$ in the Taylor series (about $x=0$)
for $G(x)$.

\bs

\noindent The next step is to show
\begin{align}
    [x^a] \,  \exp\left(\frac{r!\,q}{x^r} B_{r+1}(x) \right)
    ~~ = ~~
    \left(1 ~+~ o(1) \right) ~\times~
    [x^a] \, \exp\left(  \frac{qx}{r+1}  \right)
    \label{eq:equality-of-coeffs} 
\end{align}
Taking the difference of the LHS and RHS we rephrase \Cref{eq:equality-of-coeffs} as
\begin{align}
    [x^a] \, \left(
                     \exp\left(\frac{r!\,q}{x^r} B_{r+1}(x) \right)
                                     ~-~
                       \exp\left(  \frac{qx}{r+1}  \right)
             \right)
    ~~ = ~~
    o\left(\frac{1}{a!} \, \left( \frac{q}{r+1} \right)^a  \right) \, .
    \label{eq:coeff-is-small-o}
\end{align}
To prove \Cref{eq:coeff-is-small-o}, we use the inequality
$$
[x^a] \, F(x)  ~~ \le ~~ \frac{F(R)}{R^a} 
$$
for any positive $R$, provided $F(x)$ is entire and has real nonnegative coefficients.

\bs

\noindent Our choice for $R$ is
$R=(r+1)a/q$, which goes to zero with increasing $q$.
Here are the relevant calculations:
$$
\frac{r!\,q}{R^r} B_{r+1}(R)  ~~=~~ \frac{qR}{r+1} ~+~ O(qR^2) \, ;
$$
$$
qR^2 ~~=~~ O(a^2/q) ~~\rightarrow~~ 0 \, .
$$
Thus,
\begin{align*}
    \exp\left(\frac{r!\,q}{R^r} B_{r+1}(R) \right) ~~=~~ \exp\left(\frac{qR}{r+1} ~+~ O\left(qR^2\right)\right) ~~=~~ \exp\left(\frac{qR}{r+1}\right)\left(1 + O\left(qR^2\right)\right)
\end{align*}
$$
\exp\left(\frac{r!\,q}{R^r} B_{r+1}(R) \right) ~-~ \exp\left(\frac{qR}{r+1}\right) 
~~=~~ 
e^{qR/(r+1)} ~\times~ O(a^2/q) 
~~=~~ 
e^a \, O(a^2/q) \, .
$$

\bs

\noindent And so we have boiled \Cref{eq:coeff-is-small-o} down to
$$
e^a \, \frac{a^2}{q} \, \frac{1}{R^a} 
~~ {\buildrel ? \over = } ~~
o\left(\frac{1}{a!} \, \left( \frac{q}{r+1} \right)^a  \right) \, .
$$
Substituting $R=(r+1)a/q$,
$$
(a^2/q) \, \left( \frac{eq}{(r+1)a} \right)^a
~~ {\buildrel ? \over = } ~~
o\left(\frac{1}{a!} \, \left( \frac{q}{r+1} \right)^a  \right) \, .
$$
Dividing by $(q/(r+1))^a$, 
$$
(a^2/q) \, \frac{e^a}{a^a}
~~ {\buildrel ? \over =} ~~
o(1/a!)
$$
By Stirling, $a!e^a/a^a=O(a^{1/2})$, and so the last assertion is implied by
$$
a^{5/2}/q
~~ {\buildrel ? \over =} ~~
o(1)
$$
which is our hypothesis. This finishes the proof of \Cref{eq:SrBigq}.

\qed

\begin{lemma}
    \label{lem:Rod-approx-eq-Henn}
    Fix any integer $r\geq 1$.
    As $p\to\infty$ and $a = o(p^{2/5})$ with $a\geq 1$, then
    \begin{align*}
        \frac{(rq+a)!}{a! q! (r!)^q} \, \left(\frac{q}{r+1}\right)^a ~~\sim~~ F_r(rq+a,q)
    \end{align*}
    where $F_r$ is Hennecart's formula given in \Cref{eq:def-Hennecart-Fr}. We recall
    $$
    F_r(p,q) ~~=~~
    \frac{(rq+a)!}{q! a!} \, \left(\frac{a}{e}\right)^a \, \frac{(B_r(z_0))^q}{(z_0)^{p+1}} \,
    \sqrt{\frac{a}{\Phi^{\prime\prime}(z_0)}} \, ,
    $$
    in which $p=rq+a$ and $\Phi(z) = -p \log(z) + q \log(B_r(z))$
    and $z_0 = \Qinv$.
\end{lemma}

\proof

\bs

\noindent Here are the key calculations. 
\Cref{lem:QInvLimiting} gives a series for $Q_r^{-1}(x)$ as $x\to r$, yielding
$$
    z_0 ~~ = ~~ (r+1) \frac{a}{q} \left(1+O(a/q)\right)
$$
\begin{align}
    (1/z_0)^a ~~ = ~~ \frac{ (q/(r+1))^a }{ a^a } \, \left(1+O(a^2/q)\right)
    \label{eq:reciprocal-z0-pow-a}
\end{align}

$$
    \frac{B_r(z_0)}{(z_0)^r} ~~ = ~~ \frac{1}{r!} \exp\left( \frac{z_0}{r+1} + O(z_0^2) \right)
$$
\begin{align}
    \left(
    \frac{B_r(z_0)}{(z_0)^r} \right)^q
     ~~ = ~~ \left(\frac{1}{r!} \right)^q \, e^a \, \left(1 + O(a^2/q) \right)
    \label{eq:Br-over-z-pow-q} 
\end{align}
\Cref{lem:QTo0} gives a series for $Q_r(z)$ as $z\to 0$, yielding
\begin{align*}
\frac{B_r^{\prime}(z)}{B_r(z)} ~~ = ~~ \frac{1}{z} Q_r(z) ~~ &= ~~
\frac{r}{z} \, \left(1 + z/(r(r+1)) + O(z^2) \right) 
\end{align*}
$$
\left(\frac{B_r^{\prime}(z)}{B_r(z)}\right)^2
~~ = ~~
\frac{r^2}{z^2} \, \left( 1 + \frac{2z}{r(r+1)} + O(z^2) \right)
$$
Similarly,
\begin{align*}
\frac{B_r^{\prime\prime}(z)}{B_r(z)}
~~ &= ~~
\frac{(z^{r-2}/(r-2)! + (z^{r-1}/(r-1)! + \cdots}
     {(z^{r  }/(r  )! + (z^{r+1}/(r+1)! + \cdots} \\
~~ &= ~~
\frac{r(r-1)}{z^2} \, \frac{1+z/(r-1)+O(z^2)}{1+z/(r+1)+O(z^2)} \\
~~ &= ~~
\frac{r(r-1)}{z^2} \, \left(1 + \frac{2z}{r^2-1} + O(z^2) \right)
\end{align*}
Now we plug into
$$
\Phi^{\prime\prime}(z) ~~=~~
\frac{p}{z^2} ~+~ q\frac{B_r^{\prime\prime}(z)}{B_r(z)}
   ~-~ q \left( \frac{ B_r^{\prime}(z)}{B_r(z)}  \right)^2
$$
which will give us on the right
$$
\frac{1}{z^2} \, \left[  p ~+~ qr(r-1) \, \left(1 + 2z/(r^2-1) + O(z^2) \right) 
                       ~-~  qr^2 \, \left(1 + 2z/(r(r+1)) + O(z^2) \right)
                 \right] \, .
$$
The constant term inside the brackets is
$$
p+qr(r-1)-qr^2 ~~=~~ p-qr ~~=~~ a
$$
The coefficient of $z$ inside the brackets is
$$
2qr(r-1)/((r+1)(r-1)) ~-~ 2qr^2/(r(r+1))
~~=~~
2qr \left( 1/(r+1) ~-~ 1/(r+1) \right)
~~ = ~~
0
$$
That gives us
$$
\Phi^{\prime\prime}(z) ~~=~~
\frac{1}{z^2}\left[ a + 0 + O(qz^2) \right] 
 ~~=~~
\frac{1}{z^2}\left[ a + O(a^2/q) \right]
~~=~~ 
\frac{a}{z^2} \, \left(1 + O(a/q) \right).
$$
From this, we conclude
    \begin{align}
    \sqrt{ \frac{a}{\Phi^{\prime\prime}(z_0)} }
    ~~ = ~~
    z_0 \left(1+O(a/q)\right)
    \label{eq:sqrt-a-over-phi-dd}
\end{align}
Starting from a rearrangement of \Cref{eq:def-Hennecart-Fr},
$$
F_r(p,q) ~~=~~
\frac{(rq+a)!}{q! a!} \, \left(\frac{a}{e}\right)^a
                      \, \left(\frac{(B_r(z_0))}{(z_0)^{r}}\right)^q 
                      \, (1/z_0)^a
                      \, (1/z_0)
                      \, \sqrt{\frac{a}{\Phi^{\prime\prime}(z_0)}} \, ,
$$
we substitute from Equations (\ref{eq:reciprocal-z0-pow-a}, \ref{eq:Br-over-z-pow-q}, \ref{eq:sqrt-a-over-phi-dd}) to obtain 
$$
F_r(p,q) ~~=~~
\frac{(rq+a)!}{q! a!} \, \left(\frac{a}{e}\right)^a 
                      \, (1/r!)^q \, e^a                      
                      \, \frac{ (q/(r+1))^a }{ a^a }          
                      \, (1/z_0) 
                      \, z_0                                  
                      \, \left(1+O(a^2/q)\right)  \, .
$$
After simplification, this is found to be in agreement with \Cref{eq:SrBigq}.
\qed

\subsection{Proof of \Cref{thm:mainUniform}}

\label{sec:proof-mainUniform}
\begin{proof}
    We consider the behavior on two complementary domains.

    \emph{High $q$:}
    When $1\leq p-rq \leq p^{1/5}$, we see $q = \Theta(p)$, so $a \coloneqq p - rq = O(q^{1/5}) = o(q^{2/5})$. \Cref{lem:full-convergence-large-q} and \Cref{lem:Rod-approx-eq-Henn} together show
    \begin{align}
        S_r(rq + a, q) \sim F_r(rq + a, q)
    \end{align}
    converges as $q\to \infty$, uniformly for $a = o(q^{2/5})$. Since $p=rq+a=\Theta(q)$, this implies
    \begin{align}
        S_r(p, q) \sim F_r(p, q)
    \end{align}
    converges as $p\to\infty$ uniformly over $q$ satisfying $p - rq \leq p^{1/5}$.

    \emph{Low $q$:}
    When $p^{1/5} \leq p-rq \leq p$, we have $p-rq = \Omega(p^{1/5})$, so \Cref{lem:qz0-to-infty-except-large-q-precise} shows $qz_0 = \Omega(p^{1/5}) \to\infty$.
    By \Cref{lem:CD-unif-convergent} and \Cref{lem:CD-matches-Hennecart}, both the Hennecart (\Cref{eq:hennecart-approx-original}) and CD (\Cref{eq:CD-approx}) approximations are convergent, uniformly in $q$ on this range.

    \emph{Together:} Combine the two regimes to obtain convergence uniform for all $q$ once $p$ is big enough. This finishes
    the proof that \Cref{eq:hennecart-approx-original} is uniformly convergent for all $1 \leq q < p/r$.
\end{proof}

\section{Miscellaneous}
\label{sec:misc}

The first subsection contains a lemma needed to prove
\Cref{lem:CD-unif-convergent}.
The second gives graphs illustrating how the CD and
Hennecart approximations compare.
The last subsection gives conjectures
of two different types.

\subsection{A lemma on the zeroes of $B_r$}

Here we give a lemma essential to proving 
\Cref{lem:CD-unif-convergent}. 
This lemma fills in a detail
essential to the \cite{stirlingConnDobro} proof
of their Theorem 1.1.
To be specific, on page 30, \cite{stirlingConnDobro} writes
``Notice that for any $r \in N$ there exists $\alpha_r>0$ such that \dots{ }\!is a regular function of $z$ in the domains $|z| \leq \alpha_r$ and $R \geq 0$."
 Their proof is done by citing sources, in particular
 \cite{SV1975}, which does (only) part of the job.
 Our lemma below gives a complete proof and is adequate 
 for proving Theorem 1.1 \cite{stirlingConnDobro} 
 and our
 \Cref{lem:CD-unif-convergent}.

\begin{lemma}
    \label{lem:zeroes-of-Br}
    Fix an integer $r \geq 1$. Then there exists $\alpha_r > 0$ such that $B_r(Re^z)$ is nonzero for all $|z| \leq \alpha_r$ and all real $R > 0$.
    That is, there exists sufficiently small $\beta_r > 0$ such that $B_r$ is nonzero in the set
    \begin{align}
        \cC_r = \{ x + y i \mid x > 0, |y| \leq \beta_r x\}.
    \end{align}
\end{lemma}
\begin{proof}
    Fix $r$. Since $e^x - B_r(x) = \sum_{k=0}^{r-1} x^k/k!$ is a polynomial in $x$, we have
    \begin{align*}
        3|e^x - B_r(x)| \leq |e^x|
    \end{align*}
    for sufficiently large real $x$ (say, for all $x \geq k_r$ for a constant $k_r>0$).
    Fix $\beta_r>0$ sufficiently small such that $|1+\beta_r i|^{r-1} \leq 2$ and $(\beta_r k_r)^2 < 4(k_r + 1)$.
    Then $y^2 < 4(x+1)$ for all $x+yi \in \cC_r$ satisfying $x \leq k_r$.

    By Corollary 4.2 of \cite{SV1975}, $B_r(z)$ does not have any zeroes at $z = x + yi$ if $y^2 < 4(x+1)$, except at $z=0$. Hence $B_r(z)$ does not have any zeroes in 
    $$\cC_r \cap \{x + yi \mid x \leq k_r\}.$$ This finishes the case of $x \leq k_r$, so we now turn to $x \geq k_r$.

    \bs 

    For all $x+yi \in \cC_r$, we have $|x+yi| \leq |1 + \beta_r i| x$. Hence
    \begin{align*}
        |e^{x+yi} - B_r(x+yi)| &= \sum_{k=0}^{r-1} (x+yi)^k/k!
        \\ &\leq |1 + \beta_r i|^{r-1} \sum_{k=0}^{r-1} x^k/k!
        \\ &= |1 + \beta_r i|^{r-1} (e^x - B_r(x)).
    \end{align*}
    Thus for all $x+yi \in \cC_r \cap \{x + y i \mid x \geq k_r\}$, by choice of $\beta_r$ and $k_r$,
    \begin{align*}
        \abs{e^{x+yi} - B_r(x+yi)} &\leq 2(e^x - B_r(x))
        \\ &\leq \frac{2}{3} \abs{e^x} = \frac{2}{3} \abs{e^{x+yi}}.
    \end{align*}
    Hence in this domain,
    \begin{align*}
        \abs{B_r(x+yi)} &\geq \abs{e^{x+yi}} - \abs{e^{x+yi} - B_r(x+yi)}
        \\ &\geq \frac{1}{3} \abs{e^{x+yi}} > 0.
    \end{align*}
    Thus $B_r(x+yi)$ does not have any zeroes in $\cC_r \cap \{x + yi \mid x \geq k_r\}$.

    So far we have shown $B_r(x+yi)\neq 0$ for all $x+yi\in \cC_r$. With careful choice of $\alpha_r>0$ (specifically, defining $\alpha_r$ implicitly by $\beta_r = \max_{|z| \leq \alpha_r} \Im(e^z)/\Re(e^z)$), we get
    
    \begin{align*}
        \cC_r = \{ Re^z \mid R > 0, |z| \leq \alpha_r \},
    \end{align*}
    completing the proof.
\end{proof}

\subsection{Comparing CD and Hennecart approximations by plots}

\label{sec:comparing-CD-Hennecart}

Hennecart's formula is uniformly convergent over all $1 \leq q < p/r$, while CD's formula can only be uniformly convergent under the assumption $p-rq \to \infty$. This can be visualized by plots. We plot the relative error with the exact form of $S_r(p,q)$ for $r=2$, computed by the formula provided by Alekseyev in https://oeis.org/A008299:
\begin{align}
    \label{eq:alekseyev-stirling}
    S_2(p,q) = \sum_{i=0}^{q} (-1)^i \binom{p}{i} \curlyStirling{p-i}{q-i},
\end{align}
where $\curlyStirling{p}{q} = S_1(p,q)$ are the standard (1-associated) Stirling numbers of the second kind.

\begin{figure}[H]
    \centering
    \includegraphics[width=0.7\textwidth]{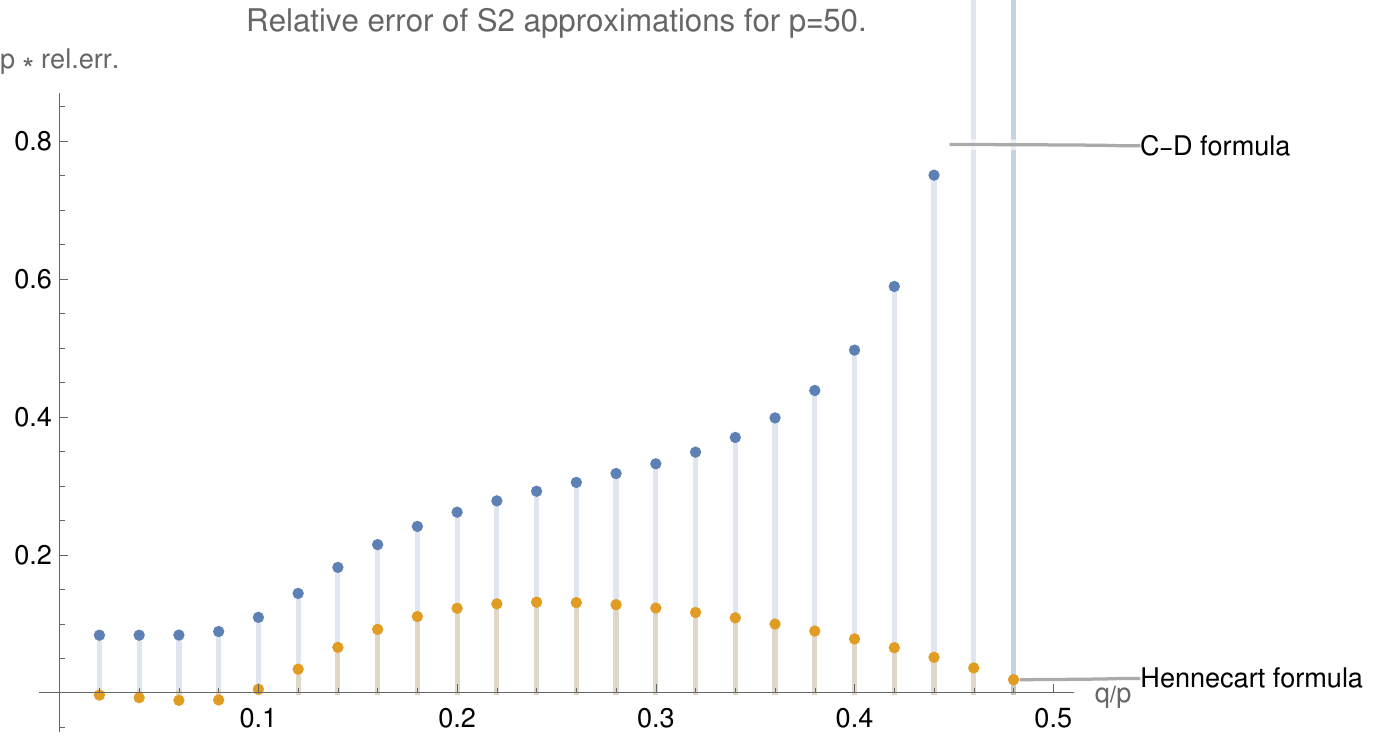}
    {}\\
    \includegraphics[width=0.7\textwidth]{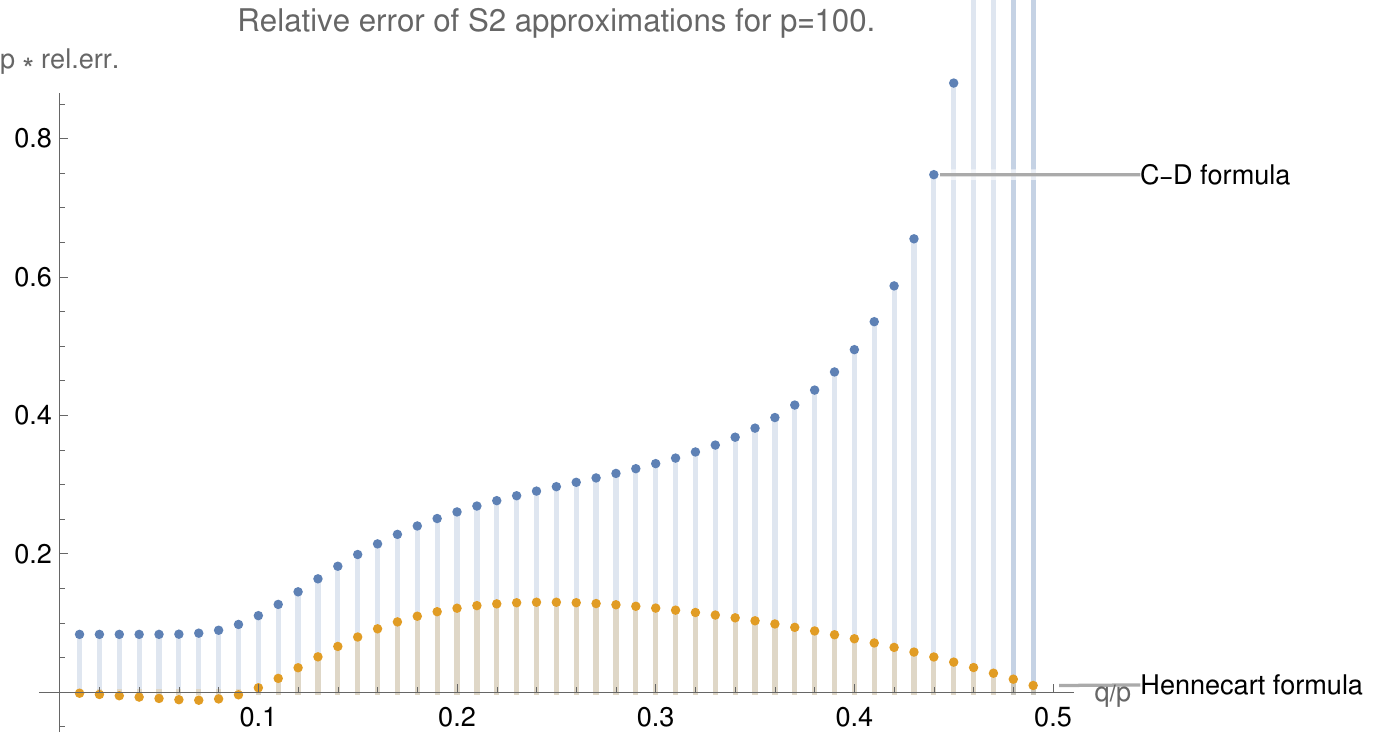}
    \caption{For $r=2$, plots of the relative error given by $\frac{\operatorname{approx}(p,q)}{S_2(p,q)} - 1$, comparing the approximation given by Hennecart (\Cref{eq:hennecart-approx-original}) and by CD (\Cref{eq:CD-approx}). The formulas are the same, except for applying the Stirling approximation for $(p-rq)!$, whose relative error is approximately given by $0.083/(p-rq)$.}
    \label{fig:CD-vs-Hennecart-relative-error}
\end{figure}

\subsection{Conjectures}

\subsubsection{A key series has positive coefficients}

We describe a conjecture which, if true, simplifies 
the proof of the asymptotic expansions given in
\cite{stirlingConnDobro}. In particular, it may simplify the proof of $|J| \to 0$.

\begin{conjecture}
    For each integer $r\geq1$, 
    the power series (about 0) of
        $$ e^{-x/(r+1)} B_r(x) $$
    has all non-negative coefficients.
\end{conjecture}
\bs 

\subsubsection{A stronger asymptotic error bound}

\begin{conjecture}
    Let $r$ be a fixed positive integer. Then as $p\to\infty$,
    \begin{align}
        S_r(p,q) = F_r(p,q)(1 + O(p^{-1}))
    \end{align}
\end{conjecture}
The conjecture is suggested by numerical plots of the relative error in \Cref{fig:Henn-relative-error}. The scaled relative error appears bounded with $p \cdot\abs{F_r / S_r - 1} < 0.16$, which suggests
    $$\abs{\frac{F_r(p,q)}{S_r(p,q)} - 1} < \frac{0.16}{p} = O(p^{-1}) .$$

    \bs 

\begin{figure}[H]
    \centering
    \includegraphics[width=0.8\linewidth]{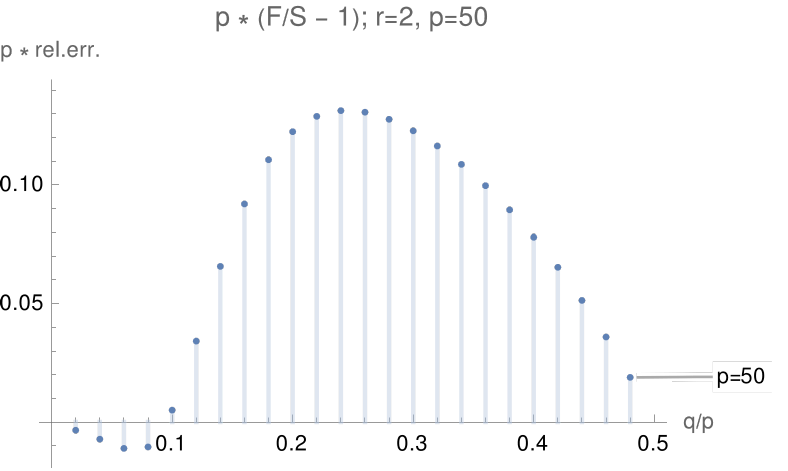}
    \includegraphics[width=0.8\linewidth]{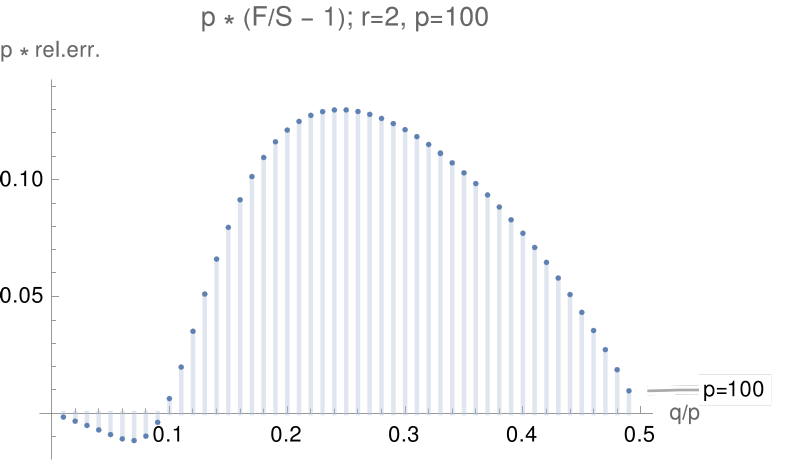}
    \caption{Scaled relative error $p \cdot (F_r / S_r - 1)$ for $r=2$. One plot shows $p=50$ and one plot shows $p=100$, but they appear to follow the same curve when plotted with respect to $q/p$.}
    \label{fig:Henn-relative-error}
\end{figure}

\clearpage
\printbibliography
\bs 

\newpage 
\begin{center}
    {\large NOT FOR PUBLICATION}
\end{center}
\thispagestyle{empty} \addtocounter{page}{-1}
\label{table-of-contents}
\tableofcontents
\clearpage

\end{document}